\documentclass{article}[12pt]

\usepackage{amsmath,amssymb,amscd,amsthm}
\usepackage{graphics}
\usepackage[dvips]{graphicx}
\usepackage{latexsym}
\usepackage{amsmath,amsfonts}
\usepackage{dsfont}
\usepackage{enumerate}
\usepackage{cite}
\usepackage{tikz}
\usepackage{hyperref}

\theoremstyle{plain}
\newtheorem{theorem}{Theorem}[section]
\newtheorem{corollary}[theorem]{Corollary}
\newtheorem{proposition}[theorem]{Proposition}
\newtheorem{lemma}[theorem]{Lemma}

\theoremstyle{definition}
\newtheorem{definition}[theorem]{Definition}
\newtheorem{remark}[theorem]{Remark}

\DeclareMathOperator{\Frac}{Frac}

\DeclareMathOperator{\Inn}{IAut}
\DeclareMathOperator{\Mult}{MAut}
\DeclareMathOperator{\Aut}{Aut}

\DeclareMathOperator{\Inv}{Inv}

\newcommand{\I}{I(X,K)}
\newcommand{\FI}{FI(X,K)}
\newcommand{\FIY}{FI(Y,K)}
\newcommand{\IL}{I(X_L,K)}
\newcommand{\FIL}{FI(X_L,K)}
\newcommand{\FILC}{FI(X_{L^c},K)}
\newcommand{\FIJFix}{FI(X_{J_3},K)}
\newcommand{\FIJFixC}{FI(X_{(J_3)^c},K)}
\newcommand{\FIj}{FI(X_j,K)}
\newcommand{\FIi}{FI(X_i,K)}
\newcommand{\Ij}{I(X_j,K)}
\newcommand{\Ii}{I(X_i,K)}

\begin{document}

\title{\textbf{Classification of involutions on finitary incidence algebras of non-connected posets}}
\author{\'{E}rica Z. Fornaroli\thanks{Corresponding author: ezancanella@uem.br} \ and Roger E. M. Pezzott\thanks{roger\_emanuel@hotmail.com} \\
\\
\textit{{\normalsize Departamento de Matem\'atica, Universidade
Estadual de Maring\'a,}} \\ \textit{{\normalsize Avenida Colombo,
5790, Maring\'a, PR, 87020-900, Brazil.}}}
\date{}

\maketitle

\begin{abstract}
Let $\FI$ be the finitary incidence algebra of a non-connected partially ordered set $X$ over a field $K$ of characteristic different from $2$. For the case where every multiplicative automorphism of $\FI$ is inner, we present necessary and sufficient conditions for two involutions on $\FI$ to be equivalent.


\end{abstract}

\maketitle

\section{Introduction}

Algebras with involution play an important role in the general theory of algebras (cf.~\cite{Jacobson,Knus}). In \cite{BFS12} was proved that the finitary incidence algebra of a partially ordered set $X$ over a field $K$, $\FI$, has an involution if and only if $X$ has an involution. Involutions on $\FI$ were first studied by Scharlau~\cite{Scharlau} and thirty years later by Spiegel~\cite{Spiegel05, Spiegel08}. Their results about the classification of involutions on $\FI$ (\cite[Theorem~2.1(b)]{Scharlau} and \cite[Theorem~1]{Spiegel08}) are incorrect as indicated in~\cite{BL}.

Let $K$ be a field of characteristic different from $2$. In the particular case where $X$ is a chain of cartinality $n$, $\FI\cong UT_n(K)$, the algebra of upper triangular matrices over $K$, and the classification of involutions on $UT_n(K)$ was given in \cite{DKL}. When $X$ has an element that is comparable to all of its elements, Brusamarello et al.~obtained the classification of involutions on $\FI$ for the case where $X$ is finite \cite{BFS11}, and later, locally finite~\cite{BFS12}. In~\cite{BFS14}, they generalized this classification for the case where $X$ is connected (not necessarily locally finite) and every multiplicative automorphism of $\FI$ is inner. This allowed us to obtain the classification of involutions on the idealization of the incidence space $\I$ over the finitary incidence algebra $\FI$ adding the hypothesis that every derivation from $\FI$ to $\I$ is inner. In this paper, we consider a field $K$ of characteristic different from $2$ and a non-connected partially ordered set $X$ such that every multiplicative automorphism of $\FI$ is inner, and we give the classification of involutions on $\FI$. For this, we use some results and ideas from \cite{BFS11, BFS14}.

Our work is organized as follows. In Section~\ref{preli} we recall some definitions and results about partially ordered sets and finitary incidence algebras, and we also present new results about automorphisms of finitary incidence algebras. In Section~\ref{SRestric} we consider  a partially ordered set $X$ written as the disjoint union of its connected components, and we give some properties of restrictions of automorphisms, anti-automorphisms and involutions of $X$ and $\FI$. In Section~\ref{SClassification} we consider a field $K$ of characteristic different from $2$ and a non-connected partially ordered set $X$ such that every multiplicative automorphism of $\FI$ is inner and we give  necessary and sufficient conditions for two involutions on $\FI$ to be equivalent via inner automorphism (Theorems~\ref{teoeqtoRholambda} and \ref{teoiffinnerJ3}). So, we use this classification to obtain the general classification of involutions on $\FI$ (Theorems~\ref{teogeralclassiinvolu} and \ref{teoiffgeral}).

\section{Preliminaries }\label{preli}

\subsection{Posets}

Let $(X, \leq)$ be a partially ordered set (poset, for short). Two elements $x,y \in X$ are \emph{comparable} if either $x \leq y$ or $y \leq x$. If an element $x_0\in X$ is comparable to every element of $X$, then it is called an \emph{all-comparable} element. Any subset $Y$ of $X$ is also a poset with the relation $\leq$ restricted to the elements of $Y$, and in this case $Y$ is said to be a \emph{subposet} of $X$. A subposet $Y$ of $X$ is a \emph{chain} if any two elements of $Y$ are comparable. On the other hand, $Y$ is a \emph{antichain} if any two distinct elements of $Y$ are not comparable.

Given $x,y \in X$, the \emph{interval} from $x$ to $y$ is the set $[x,y]= \{z \in X :$ $x \leq z \leq y\}$. If all intervals of $X$ are finite, then
$X$ is said to be \emph{locally finite}.

The elements $x, y \in X$ are \emph{connected} if for some positive integer $n$, there exists $x=x_0,x_1,\ldots, x_n=y$ in $X$ with $x_i$ and $x_{i+1}$ comparable for $i=0,1,\ldots,n-1$. It is easy to see that the connectedness of elements of $X$ is an equivalence relation whose equivalence classes  are called \emph{connected components} of $X$. Then $X$ can be written as the disjoint union of its connected components. When $X$ has only one connected component it is said to be \emph{connected}.

Let $X$ and $Y$ be posets. An \emph{isomorphism} (resp.~\emph{anti-isomorphism}) from $X$ to $Y$ is a bijective map $\varphi:X \to Y$ that satisfies the following property for any $x,y \in X$:
$$ x\leq y \Leftrightarrow \varphi(x) \leq \varphi(y) \text{ (resp.}~\varphi(y) \leq \varphi(x)\text{)}.$$
When $X=Y$, $\varphi$ is also called an \emph{automorphism} (resp.~\emph{anti-automorphism}). An anti-automorphism $\varphi:X \to X$ of order $2$ is an \emph{involution} on $X$.

Let $X$ be a poset with an involution $\lambda$. By \cite[Theorem~4.7]{BFS14}, there is a triple of disjoint subsets $(X_1, X_2, X_3)$ of $X$ with $X = X_1 \cup X_2 \cup X_3$ satisfying:
\begin{enumerate}
\item[(i)] $X_3=\{ x \in X : \lambda(x) = x\}$;
\item[(ii)] if $x \in X_1$ ($X_2$), then $\lambda(x) \in X_2$ ($X_1$);
\item[(iii)] if $x \in X_1$ ($X_2$) and $y \leq x$ ($x \leq y$), then $y \in X_1$ ($X_2$).
\end{enumerate}
In this case, $(X_1, X_2, X_3)$ is called a $\lambda$-\emph{decomposition of} $X$.

\subsection{Finitary incidence algebras}

Let $K$ be a field. Throughout the paper $K$-algebras are associative with unity. The center of a $K$-algebra $A$ is denoted by $Z(A)$ and the set of invertible elements of $A$ is denoted by $U(A)$. For each $a \in U(A)$, $\Psi_a$ denotes the inner automorphism defined by $a$, i.e., $\Psi_a:A\to A$ is such that $\Psi_a(x) = axa^{-1}$ for all $x \in A$. The group of automorphisms of $A$ is denoted by $\Aut(A)$ and the subgroup of inner automorphisms is denoted by $\Inn(A)$.

\emph{Involutions} on a $K$-algebra $A$ are ($K$-linear) anti-automorphisms of order $2$. Two involutions $\rho_1$ and $\rho_2$ are \emph{equivalent} if there exists $\phi\in\Aut(A)$ such that $\phi\circ\rho_1=\rho_2\circ\phi$.

Let $X$ be a poset and let $K$ be a field. The \emph{incidence space} $\I$ of $X$ over $K$ is the $K$-space of functions $f:X\times X\to K$ such that $f(x,y)=0$ if $x\nleq y$. Let $\FI$ be the subspace of functions $f\in \I$ such that for any $x\leq y$ in $X$, there is only a finite number of subintervals $[u,v]\subseteq[x,y]$ such that $u\neq v$ and $f(u,v)\neq 0$. Then $\FI$ is a $K$-algebra with the product (convolution)
		$$(fg)(x,y)=\sum_{x\leq z\leq y}f(x,z)g(z,y),$$
for any $f, g\in \FI$, called the \emph{finitary incidence algebra} of $X$ over $K$. Furthermore, $\I$ is an $\FI$-bimodule~\cite{KN}. If $X$ is locally finite, then $\FI=\I$ is the \emph{incidence algebra} of $X$ over $K$~\cite{Rota}. The unity of $\FI$ is the function $\delta$ defined by $\delta(x,y)=1$ if $x=y$, and $\delta(x,y)=0$ otherwise. An element $f \in \FI$ is invertible if and only if $f(x,x) \neq 0$ for all $x \in X$, by \cite[Theorem~2]{KN}. A function $f \in \FI$ is \emph{diagonal} if $f(x,y)=0$ whenever $x \neq y$. A diagonal function $f$ is \emph{constant on a set $Y$} if $f(x,x)=f(y,y)$ for all $x,y \in Y$.

\begin{proposition}\label{propcenFI}
Let $X$ be a poset and $K$ a field. Then $Z(\FI)$ is the set of all diagonal functions which are constant on each connected component of $X$. Consequently, $X$ is connected if and only if $\FI$ is a central algebra.
\end{proposition}
\begin{proof}
The proof is exactly the same as in \cite[Theorem~1.3.13]{SO97}.
\end{proof}

Let $X=\bigcup_{j \in J}X_j$ be the decomposition of the poset $X$ into its connected components. For a nonempty subset $L$ of $J$, we set $X_{L}= \bigcup_{j \in L}X_j$ and $f_{L}= f|_{X_L\times X_L}$ for each $f\in \I$. Note that $f_L\in \IL$, moreover, $f_L\in \FIL$ whenever $f\in \FI$. When $L=\{l\}$ we denote $f_{L}=f_l$. The following result is well-known and easy to prove.

\begin{proposition}
Let $X= \bigcup_{j \in J}X_j$ be the decomposition of the poset $X$ into its connected components. The map
$$\label{eqiso}
     \begin{array}{ccc}
  \phi: \I & \rightarrow & \displaystyle\prod_{j \in J}\Ij\\
    \quad  f   & \mapsto & (f_j)_{j \in J}
 \end{array}
$$
is an isomorphism of $K$-vector spaces and
\begin{equation*}\label{iso_FI_prod_FIX_j}
\phi|_{\FI}: \FI  \rightarrow \prod_{j \in J} \FIj
\end{equation*}
is an isomorphism of $K$-algebras.
\end{proposition}

When convenient, we will identify $f\in \I$ with $(f_j)_{j \in J}\in \prod_{j \in J} \Ij$.

The next corollary follows directly from the proposition above.

\begin{corollary}\label{cor1}
Let $X=\bigcup_{j \in J}X_j$ be the decomposition of the poset $X$ into its connected components. Let $L$ be a nonempty subset of $J$ and $L^c=J-L$. Then, for any $f,g\in \I$,
\begin{enumerate}
\item[(i)] $(f+g)_L=f_L+g_L$;
\item[(ii)] $(af)_L=af_L$ for all $a\in K$;
\item[(iii)] $(fg)_L=f_Lg_L$ if $f,g\in \FI$;
\item[(iv)] $f \in U(\FI) \Leftrightarrow f_L \in U(\FIL)$ and $f_{L^c} \in U(\FILC)$. In this case, $(f^{-1})_L=(f_L)^{-1}$;
\item[(v)] $f \in Z(\FI) \Leftrightarrow f_L \in Z(\FIL)$ and $f_{L^c} \in Z(\FILC)$.
\end{enumerate}
\end{corollary}

\begin{corollary}\label{cor2}
Let $X=\bigcup_{j \in J}X_j$ be the decomposition of the poset $X$ into its connected components. Let $L$ be a nonempty subset of $J$ and $u,v\in U(\FI)$. Then
\begin{enumerate}
\item[(i)] $[\Psi_u(f)]_L=\Psi_{u_L}(f_L)$ for all $f\in \FI$;
\item[(ii)] $\Psi_u =\Psi_v \Leftrightarrow \Psi_{u_j} =\Psi_{v_j} \ \forall j \in J \Leftrightarrow \Psi_{u_L} =\Psi_{v_L} \text{ and } \Psi_{u_{L^c}} =\Psi_{v_{L^c}}$.
\end{enumerate}
\end{corollary}
\begin{proof}
It follows from the previous corollary and \cite[Proposition 2.2]{ER2}.
\end{proof}

An element $\sigma \in \I$ such that $\sigma(x,y)\neq 0$ for all $x \leq y$, and  $\sigma(x,y)\sigma(y,z)=\sigma(x,z)$  whenever $x\leq y\leq z$ is called \emph{multiplicative}. For each map $h: X \to K^*$, we define a \emph{fractional} element $\tau_h \in \I$ by $\tau_h(x,y)= \dfrac{h(x)}{h(y)}$ for all $x \leq y$ in $X$. Its easy to check that any fractional element is multiplicative.

\begin{lemma}\label{lemmaiffmultfrac}
Let $X=\bigcup_{j \in J}X_j$ be the decomposition of the poset $X$ into its connected components and let $\sigma \in \I$. Then $\sigma$ is fractional (multiplicative) if and only if $\sigma_j\in \Ij$ is fractional (multiplicative) for all $j \in J$.
\end{lemma}
\begin{proof}
Clearly, $\sigma$ is multiplicative if and only if $\sigma_j$ is multiplicative for all $j \in J$, because if $x \leq y \leq z$, then $x,y$ and $z$ are in a same connected component of $X$.

If $\sigma = \tau_h$ for some $h:X\to K^*$, then $\sigma_j = \tau_{h_j}$ for all $j \in J$, where $h_j= h|_{X_j}$. Conversely, if for each $j\in J$ there is $g_j:X_j\to K^*$ such that $\sigma_j= \tau_{g_j}$, then $\sigma = \tau_g$ where $g:X\to K^*$ is given by $g(x) = g_j(x)$ if $x \in X_j$.
\end{proof}

For each multiplicative element $\sigma \in \I$ consider $M_\sigma:\FI\to \FI$ defined by $M_{\sigma}(f)(x,y)=\sigma(x,y)f(x,y)$, for all $f \in \FI$ and $x,y \in X$. Then $M_{\sigma}$ is an automorphism which is said to be a \emph{multiplicative} automorphism. We denote the group of all multiplicative automorphisms of $\FI$ by $\Mult(\FI)$. The subset $\Frac(\FI)= \{M_{\sigma}: \sigma$ is fractional$\}$ is a subgroup of $\Mult(\FI)$ and its elements are called \emph{fractional} automorphisms.

As in the case where $X$ is locally finite, we have the following relationship between the inner, multiplicative and fractional automorphisms of $\FI$.

\begin{proposition}\label{InnMultFrac}
For any field $K$ and any poset $X$,
$$\Inn(\FI) \cap \Mult(\FI)= \Frac(\FI).$$
\end{proposition}
\begin{proof}
It is analogous to the proof of \cite[Proposition~7.3.3]{SO97}.
\end{proof}

\begin{proposition}\label{MultMult}
Let $X=\bigcup_{j \in J}X_j$ be the decomposition of the poset $X$ into its connected components. The multiplicative automorphism $M_\sigma$ of $\FI$ is inner if and only if the multiplicative automorphism $M_{\sigma_j}$ of $\FIj$ is inner for each $j \in J$.
\end{proposition}
\begin{proof}
It follows directly from Lemma~\ref{lemmaiffmultfrac} and Proposition~\ref{InnMultFrac}.
\end{proof}

\begin{corollary}\label{coroiffMultInner}
Let $X=\bigcup_{j \in J}X_j$ be the decomposition of the poset $X$ into its connected components. Then $\Mult(\FI) \subseteq \Inn(\FI)$ if and only if $\Mult(\FIj) \subseteq \Inn(\FIj)$ for all $j \in J$.
\end{corollary}
\begin{proof}
Suppose $\Mult(\FI) \subseteq \Inn(\FI)$. Let $i \in J$ and $\sigma \in \Ii$ be a multiplicative element. By Lemma~\ref{lemmaiffmultfrac}, $\tau \in \I$ defined by
$$\tau(x,y) = \begin{cases}
\sigma(x,y) & \text{if } x \leq y \in X_i\\
1 & \text{if } x \leq y \in X_j, j\neq i \\
0 & \text{if } x\nleq y
\end{cases}$$
is multiplicative. Thus, $M_\tau \in \Inn(\FI)$, therefore $M_{\tau_i}\in \Inn(\FIi)$, by Proposition~\ref{MultMult}. Since $M_{\sigma}=M_{\tau_i}$, then  $M_{\sigma}\in \Inn(\FIi)$.

The converse follows directly from Proposition~\ref{MultMult}.
\end{proof}

\begin{corollary}\label{corolocalfinito}
Let $X=\bigcup_{j \in J}X_j$ be the decomposition of the poset $X$ into its connected components. If $X_j$ has an all-comparable element for each $j \in J$, then $\Mult(\FI) \subseteq \Inn(\FI)$.
\end{corollary}
\begin{proof}
It follows directly from Corollary~\ref{coroiffMultInner} and \cite[Proposition~2.4]{BFS14}.
\end{proof}

Let $X$ and $Y$ be posets and let $\varphi: X\to Y$ be an isomorphism (resp.~anti-isomorphism). Then $\varphi$ induces an isomorphism (resp.~anti-isomorphism) $\widehat{\varphi}$ (resp. $\rho_\varphi$) from $\FI$ to $\FIY$ given by
\begin{center}
$\widehat{\varphi}(f)(x,y)=f(\varphi^{-1}(x), \varphi^{-1}(y))$  (resp.~$\rho_\varphi(f)(x, y)=f(\varphi^{-1}(y), \varphi^{-1}(x))$),
\end{center}
for all $f \in \FI$ and $x, y \in Y$. In particular, when $X =Y$, $\widehat{\varphi}$ is an automorphism (resp.~anti-automorphism) of $\FI$. Moreover, if $\varphi$ is an involution, so is $\rho_\varphi$.

The theorem below follows from \cite[Lemma~3]{AutK} for automorphisms and it is \cite[Theorem~3.5]{BFS14} for anti-automorphisms and involutions.

\begin{theorem} \label{DecoAutFIP}
Let $X$ be a poset and let $K$ be a field. If $\Phi$ is an automorphism (anti-automorphism, involution) of $\FI$, then $\Phi = \Psi \circ M \circ \phi$, where $\Psi \in \Inn(\FI)$, $M\in \Mult(\FI)$ and $\phi$ is the automorphism (anti-automorphism, involution) induced by an automorphism (anti-automorphism, involution) of $X$.
\end{theorem}

\begin{remark}\label{remarkDecomposition}
The automorphism (anti-automorphism, involution) $\phi$ in the theorem above is equal to $\widehat{\varphi}$ (resp.~$\rho_\varphi$), where $\varphi$  is the automorphism (anti-automorphism, involution) of $X$ induced by $\Phi$. (See proofs of \cite[Lemma~3]{AutK} and \cite[Theorem~3.5]{BFS14}).
\end{remark}

Let $X=\bigcup_{j \in J}X_j$ be the decomposition of the poset $X$ into its connected components. If $X_j$ has an all-comparable element for each $j \in J$, then $\Mult(\FI) \subseteq \Inn(\FI)$, by Corollary~\ref{corolocalfinito}. In this case, the decomposition of $\Phi$ given by Theorem~\ref{DecoAutFIP} takes the simplest form $\Phi = \Psi \circ \phi$. However, the existence of an all-comparable element in $X_j$ for each $j \in J$ is not necessary to ensure that every multiplicative automorphism of $FI(X,K)$ is inner. For example, if $X_j$ is a finite tree for all $j\in J$, then $\Mult(\FI) \subseteq \Inn(\FI)$, by \cite[Corollary~6]{BMUL} and Corollary~\ref{coroiffMultInner}. Necessary and sufficient conditions for a multiplicative automorphism of $\FI$ to be inner were given in~\cite{BMUL} for the case when $X$ is finite and connected.

\section{Restrictions of automorphisms, anti-automorphisms and involutions} \label{SRestric}

From now on we consider a non-connected poset $X$ such that $\Mult(\FI) \subseteq \Inn(\FI)$ and we write $X=\bigcup_{j \in J}X_j$ its decomposition into connected components. In this case, if $\Phi$ is an automorphism (anti-automorphism, involution) of $\FI$ that induces the automorphism (anti-automorphism, involution) $\varphi$ of $X$, then there is $u\in U(\FI)$ such that $\Phi=\Psi_u \circ \widehat{\varphi}$ ($\Phi=\Psi_u \circ \rho_\varphi$), by Theorem~\ref{DecoAutFIP} and Remark~\ref{remarkDecomposition}.

\begin{remark}\label{rest_maps_posets}
Let $\varphi:X \to X$ be an automorphism (anti-automorphism, involution).
\begin{enumerate}
\item[(i)] For each $i\in J$ there exists $j \in J$ such that $\varphi(X_i)=X_j$.
\item[(ii)] If $L$ is a nonempty subset of $J$ such that $\varphi(X_L)=X_L$, then $\varphi|_{X_L}:X_L \to X_L$ is an automorphism (anti-automorphism, involution) which will be denoted by $\varphi_L$, and $(\varphi_{L})^{-1}=(\varphi^{-1})_{L}$. When $L=\{l\}$ we just write $\varphi_L=\varphi_l$.
\end{enumerate}
\end{remark}



\begin{definition}\label{def1}
Let $\alpha$ and $\lambda$ be an automorphism and an anti-automorphism of $X$, respectively, $L$ a nonempty subset of $J$ such that $\alpha(X_L)=X_L$ and $\lambda(X_L)=X_L$, and $u\in U(\FI)$. For $\Phi=\Psi_u\circ \widehat{\alpha}$ and $\rho=\Psi_u\circ \rho_{\lambda}$ we define
$$ \Phi_L:=\Psi_{u_L}\circ \widehat{\alpha_L} \ \text{ and } \ \rho_L:=\Psi_{u_L}\circ \rho_{\lambda_L}.$$
When $L=\{l\}$ we write $\Phi_L= \Phi_l$ and $\rho_L=\rho_l$.
\end{definition}

\begin{proposition}\label{prop2}
Let $\alpha$ and $\lambda$ be an automorphism and an anti-automorphism of $X$, respectively, and let $L$ be a nonempty subset of $J$ such that $\alpha(X_L)=X_L$ and $\lambda(X_L)=X_L$. Let $\Phi$ be an automorphism of $\FI$ that induces $\alpha$ on $X$ and $\rho$ an anti-automorphism of $\FI$ that induces $\lambda$ on $X$. Then

\begin{enumerate}
\item[(i)]$[\rho(f)]_L= \rho_L(f_L)$ for all $f\in \FI$;
\item[(ii)]$[\Phi(f)]_L= \Phi_L(f_L)$ for all $f\in \FI$;
\item[(iii)]$[\Phi \circ \rho]_L = \Phi_L \circ \rho_L$ and $[\rho\circ\Phi]_L = \rho_L \circ \Phi_L$.
\end{enumerate}
\end{proposition}
\begin{proof}
Suppose $\Phi= \Psi_v \circ \widehat{\alpha}$ and $\rho= \Psi_u \circ \rho_\lambda$, for some $v,u \in U(\FI)$.

(i) Let $f\in \FI$. Firstly, if $\lambda(X_L)=X_L$, then for all $x\leq y$ in $X_L$ we have
\begin{align*}
  [\rho_\lambda(f)]_L(x,y) & = \rho_\lambda(f)(x,y)=f(\lambda^{-1}(y),\lambda^{-1}(x))=f_L(\lambda^{-1}(y),\lambda^{-1}(x))\\
    & = f_L((\lambda^{-1})_L(y),(\lambda^{-1})_L(x))= f_L((\lambda_L)^{-1}(y),(\lambda_L)^{-1}(x))\\
    & = \rho_{\lambda_L}(f_L)(x,y).
\end{align*}
Thus, $[\rho_\lambda(f)]_L= \rho_{\lambda_L}(f_L)$. Therefore,
$$[\rho(f)]_L=[\Psi_u(\rho_{\lambda}(f))]_L=\Psi_{u_L}([\rho_{\lambda}(f)]_L)=\Psi_{u_L}(\rho_{\lambda_L}(f_L))=\rho_L(f_L),$$
 by Corollary~\ref{cor2} (i) and Definition~\ref{def1}.

(ii) It is similar to (i).

(iii) Note that $\Phi \circ \rho$ induces the anti-automorphism $\alpha\circ \lambda$ of $X$, because
$$\Phi \circ \rho= \Psi_v \circ \widehat{\alpha} \circ \Psi_u \circ \rho_\lambda = \Psi_{v} \circ \Psi_{\widehat{\alpha}(u)} \circ \widehat{\alpha} \circ \rho_\lambda = \Psi_{v\widehat{\alpha}(u)} \circ \rho_{\alpha \circ\lambda}.$$
Moreover, ($\alpha \circ \lambda) (X_L)=X_L$. Given $g \in \FIL$, consider $f \in \FI$ such that $f_L=g$. By (i) and (ii) we have
$$(\Phi_L \circ \rho_L)(g)= \Phi_L(\rho_L(f_L))= [\Phi(\rho(f))]_L =  [\Phi \circ \rho]_L(f_L)=[\Phi \circ \rho]_L(g).$$
Therefore, $[\Phi \circ \rho]_L = \Phi_L \circ \rho_L$. Analogously, $[\rho\circ\Phi]_L = \rho_L \circ \Phi_L$.
\end{proof}

\begin{proposition}\label{prop3}
Let $\lambda$ be an involution on $X$, $L$ a nonempty subset of $J$ such that $\lambda(X_L)=X_L$, and $u\in U(\FI)$. Let $\rho=\Psi_u\circ \rho_{\lambda}$. Then $\rho$ is an involution on $\FI$ if and only if $\rho_L=\Psi_{u_L}\circ \rho_{\lambda_L}$ is an involution on $\FIL$ and $\rho_{L^c}=\Psi_{u_{L^c}}\circ \rho_{\lambda_{L^c}}$ is an involution on $\FILC$.
\end{proposition}
\begin{proof}
We have
\begin{align*}
  \rho \text{ is an involution} & \Leftrightarrow \Psi_u=\Psi_{\rho_{\lambda}(u)} \text{ by \cite[Proposition~2.2]{ER2}}\\
    & \Leftrightarrow \Psi_{u_L}=\Psi_{[\rho_{\lambda}(u)]_L} \text{ and } \Psi_{u_{L^c}}=\Psi_{[\rho_{\lambda}(u)]_{L^c}} \text{ by Corollary~\ref{cor2} (ii)}\\
  & \Leftrightarrow \Psi_{u_L}=\Psi_{\rho_{\lambda_L}(u_L)} \text{ and } \Psi_{u_{L^c}}=\Psi_{\rho_{\lambda_{L^c}}(u_{L^c})}  \text{ by Proposition~\ref{prop2} (i)}\\
  & \Leftrightarrow \rho_L \text{ and } \rho_{L^c} \text{ are involutions by \cite[Proposition~2.2]{ER2}}.
\end{align*}
\end{proof}

\begin{proposition}\label{cor3}
Let $\lambda$ be an involution on $X$ such that $\lambda(X_j)=X_j$ for all $j\in J$, and $u\in U(\FI)$. Let $\rho=\Psi_u\circ \rho_{\lambda}$. Then $\rho$ is an involution on $\FI$ if and only if $\rho_j=\Psi_{u_j}\circ \rho_{\lambda_j}$ is an involution on $\FIj$ for all $j\in J$.
\end{proposition}
\begin{proof}
It is analogous to the proof of Proposition~\ref{prop3}.
\end{proof}

\section{Classification of involutions}\label{SClassification}

We recall that the involutions $\rho_1$ and $\rho_2$ on $\FI$ are equivalent if there exists $\phi\in\Aut(\FI)$ such that $\phi\circ\rho_1=\rho_2\circ\phi$. In this section we consider a field $K$ of characteristic different from $2$ and a non-connected poset $X$ such that $\Mult(\FI)\subseteq \Inn(\FI)$ and we give necessary and sufficient conditions for two involutions on $\FI$ to be equivalent. For this, we use some results and ideas from \cite{BFS11} and \cite{BFS14}.

\subsection{Classification of involutions via inner automorphisms}

We start by considering the classification via inner automorphisms. Let $\rho_1$ and $\rho_2$ be involutions on $\FI$. If there exists $\Psi\in\Inn(\FI)$ such that $\Psi\circ\rho_1=\rho_2\circ\Psi$, then $\rho_1$ and $\rho_2$ induce the same involution on $X$, by \cite[Theorem~5.1]{BFS14}. Thus, two involutions on $\FI$ that induce different involutions on $X$ are not equivalent via inner automorphisms. For that, in this subsection we fix an involution $\lambda$ on $X$ and give necessary and sufficient conditions for two involutions on $\FI$ that induce $\lambda$ on $X$ to be equivalent via inner automorphisms.

We denote by $\Inv_{\lambda}(\FI)$ the set of all involutions on $\FI$ that induce $\lambda$ on $X$.



\begin{theorem}\label{teoifffixPL}
Let $\rho ,\eta  \in \Inv_{\lambda}(\FI)$ and let $L$ be a nonempty subset of $J$ such that $\lambda(X_L)=X_L$. Consider the involutions $\rho_L$ and $\eta_L$ on $\FIL$, and $\rho_{L^c}$ and $\eta_{L^c}$ on $\FILC$. Then $\rho$ and $\eta$ are equivalent via inner automorphism, if and only if $\rho_L$ and $\eta_L$ are equivalent via inner automorphism, and $\rho_{L^c}$ and $\eta_{L^c}$ are equivalent via inner automorphism.
\end{theorem}
\begin{proof}
Note that if $\lambda(X_{L})=X_{L}$, then  $\lambda(X_{L^c})=X_{L^c}$. Therefore, if  $\Psi_u\circ \rho =\eta \circ \Psi_u$ for some $u \in U(\FI)$, then $\Psi_{u_L}\circ \rho_L =\eta_L \circ \Psi_{u_L}$ and $\Psi_{u_{L^c}}\circ \rho_{L^c} =\eta_{L^c} \circ \Psi_{u_{L^c}}$, by Proposition~\ref{prop2}~(iii).

Conversely,  let $w_1 \in U(\FIL)$ and $w_2 \in U(\FILC)$ such that $\Psi_{w_1}\circ \rho_L =\eta_L \circ \Psi_{w_1}$ and $\Psi_{w_2}\circ \rho_{L^c} =\eta_{L^c} \circ \Psi_{w_2}$. Let $w \in \FI$ such that $w_L=w_1$ and $w_{L^c}=w_2$. Then $w\in U(\FI)$. By Corollary~\ref{cor2}~(i) and Proposition~\ref{prop2}~(i), we have
\begin{align}\label{eqteo41}
[(\Psi_w \circ \rho)(f)]_L & = \Psi_{w_L}([\rho(f)]_L)= \Psi_{w_1}(\rho_L(f_L)) =\eta_L (\Psi_{w_1}(f_L))=\eta_L (\Psi_{w_L}(f_L))\nonumber\\
                           & =\eta_L( [\Psi_{w}(f)]_L)= [(\eta\circ \Psi_w)(f)]_L
\end{align}
for all $f \in \FI$. Analogously, $[(\Psi_w \circ \rho)(f)]_{L^c}=[(\eta \circ \Psi_{w})(f)]_{L^c}$ for all $f \in \FI$. Therefore, $(\Psi_w \circ \rho)(f) = (\eta \circ \Psi_w)(f)$ for all $f \in \FI$, that is, $\Psi_w \circ \rho = \eta \circ \Psi_w.$
\end{proof}

Consider the following relation defined on the set $J$ (such that $X=\bigcup_{j \in J}X_j$):
$$i\preceq j\Leftrightarrow X_i=X_j.$$
Then $(J,\preceq)$ is a poset (antichain). Note that the involution $\lambda:X \to X$ induces an involution $\overline{\lambda}: J \to J$ such that, given $j\in J$, one has $\lambda(X_j)= X_{\overline{\lambda}(j)}$, by Remark~\ref{rest_maps_posets} (i). Let $(J_1,J_2,J_3)$ be a $\overline{\lambda}$-decomposition of $J$. Then
\begin{equation} \label{eqj3}
     J_3= \{j \in J : \overline{\lambda}(j)=j\}=\{j \in J : \lambda(X_j)=X_j\}.
\end{equation}

\begin{lemma}\label{lemmacentralunit}
If $\rho=\Psi_u \circ \rho_\lambda$ is an involution on $\FI$, then there is a central unit $v$ of $\FI$ such that $\rho_\lambda(u)=vu$. Moreover, for each $j \in J$, there is $k_j\in K^*$ such that $v(x,x)=k_j$ for all $x\in X_j$, and $k_jk_{\overline{\lambda}(j)}=1$.
\end{lemma}
\begin{proof}
By (i) and (ii) of \cite[Proposition~2.2]{ER2}, there is a central unit $v \in \FI$ such that $\rho_\lambda(u)=vu$. So
$$u= \rho_{\lambda}^2(u)= \rho_{\lambda}(vu)= \rho_{\lambda}(u)\rho_{\lambda}(v)= vu\rho_{\lambda}(v)=uv\rho_{\lambda}(v)$$
and, consequently, $v\rho_{\lambda}(v)=\delta$. Since $v\in Z(\FI)$, then $v$ is diagonal and for each $j\in J$ there is $k_j\in K^*$ such that $v(x,x)=k_j$ for all $x\in X_j$, by Proposition~\ref{propcenFI}. Thus, if $x \in X_j$, then $\lambda(x) \in \lambda(X_j)=X_{\overline{\lambda}(j)}$ and so
$$1=\delta(x,x)=(v\rho_\lambda(v))(x,x)=v(x,x)v(\lambda(x),\lambda(x))=k_jk_{\overline{\lambda}(j)}.$$
\end{proof}

In order to classify the involutions on $\FI$ that induce $\lambda$ on $X$, via inner automorphisms, we consider the cases $J_3 =\emptyset$ and $J_3\neq \emptyset$. We start by considering $J_3 =\emptyset$ or, equivalently, $\lambda(X_j) \neq X_j$ for all $j \in J$.

\begin{theorem}\label{teoeqtoRholambda}
Let $\lambda$ be an involution on $X$ such that $J_3 = \emptyset$. Every $\rho \in \Inv_{\lambda}(\FI)$ is equivalent to $\rho_\lambda$ via inner automorphism.
\end{theorem}
\begin{proof}
Let $u\in U(\FI)$ such that $\rho=\Psi_u \circ \rho_\lambda$. By Lemma~\ref{lemmacentralunit}, there is a central unit $v$ of $\FI$ such that $\rho_\lambda(u)=vu$ and $k_jk_{\overline{\lambda}(j)}=1$, where $v(x,x)=k_j$ for all $x\in X_j$, for each $j \in J$. Define $w \in \FI$ by
\begin{align}\label{eqh}
w(x,y)=\begin{cases}
k_j & \text{if } x=y \in X_j \text{ and } j \in J_1 \\
1 & \text{if } x=y \in X_j \text{ and } j \in J_2 \\
0 & \text{if } x\neq y
\end{cases}.
\end{align}
Then $w$ is a central unit of $\FI$ and
\begin{align}\label{eqrhoH}
    \rho_\lambda(w)(x,y) = w(\lambda(y),\lambda(x)) & = \begin{cases}
k_j & \text{if } \lambda(x)=\lambda(y) \in X_j \text{ and } j \in J_1 \\
1 & \text{if } \lambda(x)=\lambda(y) \in X_j \text{ and } j \in J_2 \\
0 & \text{if } x\neq y
\end{cases} \nonumber \\
& = \begin{cases}
k_j & \text{if } x=y \in X_{\overline{\lambda}(j)} \text{ and } j \in J_1 \\
1 & \text{if } x=y \in X_{\overline{\lambda}(j)}\text{ and } j \in J_2 \\
0 & \text{if } x\neq y
\end{cases}.
\end{align}
Since $\rho_\lambda(w)$ and $v$ are diagonal, so is $\rho_\lambda(w)v$, and by \eqref{eqh} and \eqref{eqrhoH}, we have
\begin{align*}
    (\rho_\lambda(w)v)(x,x) & = \rho_\lambda(w)(x,x)v(x,x)\\
     & = \begin{cases}
k_j k_{\overline{\lambda}(j)} & \text{if } x\in X_{\overline{\lambda}(j)} \text{ and } j \in J_1 \\
k_{\overline{\lambda}(j)} & \text{if } x \in X_{\overline{\lambda}(j)}\text{ and } j \in J_2
\end{cases}\\
 & = \begin{cases}
1 & \text{if } x\in X_{\overline{\lambda}(j)} \text{ and } j \in J_1 \\
k_{\overline{\lambda}(j)} & \text{if } x\in X_{\overline{\lambda}(j)}\text{ and } j \in J_2
\end{cases}\\
 & =\begin{cases}
1 & \text{if } x \in X_j \text{ and } j \in J_2 \\
k_j & \text{if } x\in X_j \text{ and } j \in J_1
\end{cases} \\
& = w(x,x)
\end{align*}
for all $x\in X$. Therefore, $\rho_\lambda(w)v=w$.

Let $u_1=wu \in U(\FI)$. Then $\rho_\lambda(u_1)=\rho_{\lambda}(w)\rho_{\lambda}(u)=\rho_{\lambda}(w)vu=wu=u_1$. Moreover, $\Psi_u = \Psi_{u_1}$, by \cite[Proposition~2.2]{ER2}. Thus, $\rho=\Psi_{u_1}\circ \rho_{\lambda}$ with $\rho_\lambda(u_1)=u_1$. Since $\lambda(X_j)\neq X_j$ for all $j\in J$, then $\lambda(x)\neq x$ for all $x\in X$, by Remark~\ref{rest_maps_posets}~(i). So, by \cite[Lemma~5.3]{BFS14}, there exists $v_1 \in U(\FI)$ such that $u_1=v_1\rho_\lambda(v_1)$. Therefore, $\rho=\Psi_{v_1\rho_{\lambda}(v_1)}\circ \rho_\lambda$, thus $\rho$ and $\rho_\lambda$ are equivalent via inner automorphism, by \cite[Proposition~2.2]{ER2}.
\end{proof}

Now, we suppose $J_3\neq \emptyset$. By \eqref{eqj3},
\begin{align}\label{eqfixPJ3}
\lambda(X_{J_3})=\lambda\left(\displaystyle\bigcup_{j \in J_3}X_j\right)=\displaystyle\bigcup_{j \in J_3}\lambda(X_j)=\displaystyle\bigcup_{j \in J_3}X_j=X_{J_3}.
\end{align}

\begin{theorem}\label{teoiffinnerJ3}
Let $\lambda$ be an involution on $X$ such that $J_3\neq\emptyset$. Let $\rho ,\eta  \in \Inv_{\lambda}(\FI)$. The following statements are equivalent:
\begin{enumerate}
    \item[(i)] $\rho$ and $\eta$ are equivalent via inner automorphism.
    \item[(ii)] $\rho_{J_3}$ and $\eta_{J_3}$ are equivalent via inner automorphism.
    \item[(iii)] $\rho_{j}$ and $\eta_{j}$ are equivalent via inner automorphism, for all $j \in J_3$.
\end{enumerate}

\end{theorem}
\begin{proof}

(i) $\Rightarrow$ (iii)
By \eqref{eqj3}, $\lambda(X_j)=X_j$ for all $j\in J_3$. Thus, if $\rho$ and $\eta$ are equivalent via inner automorphism, then $\rho_{j}$ and $\eta_{j}$ are equivalent via inner automorphism for all $j \in J_3$, by Theorem~\ref{teoifffixPL}.

(iii) $\Rightarrow$ (ii)  Let $w_j \in \FIj$ such that  $\Psi_{w_j}\circ \rho_{j}=\eta_{j} \circ \Psi_{w_j}$ for all $j \in J_3$. Let $w\in U(\FIJFix)$ such that $w=(w_j)_{j\in J_3}\in \prod_{j\in J_3}\FIj$. As in \eqref{eqteo41}, we can show that $\Psi_{w}\circ \rho_{J_3}=\eta_{J_3} \circ \Psi_{w}$ and, therefore, $\rho_{J_3}$ and $\eta_{J_3}$ are equivalent via inner automorphism.

(ii) $\Rightarrow$ (i) Suppose $\rho_{J_3}$ and $\eta_{J_3}$ are equivalent via inner automorphism. Since $\lambda(X_{J_3})=X_{J_3}$, then $\lambda(X_{(J_3)^c})=X_{(J_3)^c}$. Thus $\lambda(X_j) \neq X_j$ for all $j \in (J_3)^c$. By Theorem~\ref{teoeqtoRholambda}, the involutions $\rho_{(J_3)^c}$ and $\eta_{(J_3)^c}$ on $\FIJFixC$ are equivalent via inner automorphism. Therefore, $\rho$ and $\eta$ are equivalent via inner automorphism, by Theorem~\ref{teoifffixPL}.

\end{proof}

 Denote by $\thickapprox$ the equivalence of elements of $\Inv_{\lambda}(\FI)$ via inner automorphisms, that is, for $\rho, \eta \in \Inv_{\lambda}(\FI)$,
\begin{equation*}\label{eqeqrelation}
\rho \thickapprox \eta \Leftrightarrow \Psi \circ \rho =\eta \circ \Psi \text{ for some } \Psi \in \Inn(\FI).
\end{equation*}
As usual, we denote the quotient set by $\dfrac{\Inv_\lambda(\FI)}{\thickapprox}$ and the equivalence class of $\rho\in \Inv_\lambda(\FI)$ by $\overline{\rho}$.

\begin{lemma}\label{bij_classes}
The following map $\Omega$ is a bijection:
\begin{equation*}\label{eqbijection}
\begin{array}{cccc}
\Omega: & \dfrac{\Inv_\lambda(\FI)}{\thickapprox} & \rightarrow  & \displaystyle \prod_{j \in J_3}\frac{\Inv_{\lambda_j}(\FIj)}{\thickapprox}\\
  & \overline{\rho} & \mapsto & (\overline{\rho_j})_{j \in J_3}
\end{array}.
\end{equation*}
\end{lemma}
\begin{proof}
By Theorem~\ref{teoiffinnerJ3}, $\Omega$ is well-defined and injective. For each $j\in J_3$, let $\phi_j \in \Inv_{\lambda_j}(\FIj)$. Since $\Mult(\FIj)\subseteq \Inn(\FIj)$, by Corollary~\ref{coroiffMultInner},  there is $u_j\in U(\FIj)$ such that $\phi_{j}= \Psi_{u_j}\circ \rho_{\lambda_j}$, by Theorem~\ref{DecoAutFIP}. Let $v \in U(\FI)$ such that $v_j=u_j$ if $j \in J_3$, and $v_j=\delta_j$ if $j\not\in J_3$. Consider the anti-automorphism $\Phi= \Psi_{v} \circ \rho_{\lambda}$ of $\FI$. By Definition~\ref{def1}, for each $j\in J_3$, $\Phi_{j}= \Psi_{v_{j}} \circ \rho_{\lambda_{j}}= \Psi_{u_{j}} \circ \rho_{\lambda_{j}}=\phi_{j}$. Thus, since $\lambda(X_j)=X_j$ for all $j\in J_3$, then  $\Phi_{J_3}$ is an involution on $\FIJFix$, by Proposition~\ref{cor3}. On the other hand, $\Phi_{(J_3)^c}= \Psi_{v_{(J_3)^c}} \circ \rho_{\lambda_{(J_3)^c}}=\rho_{\lambda_{(J_3)^c}}$ which is an involution on $\FIJFixC$. By Proposition~\ref{prop3}, $\Phi$ is an involution on $\FI$. Therefore, $\Phi\in \Inv_{\lambda}(\FI)$ and $\Omega(\overline{\Phi})=(\overline{\Phi_j})_{j \in J_3}=(\overline{\phi_j})_{j \in J_3}$.
\end{proof}

Let $(P_1,P_2,P_3)$ be a $\lambda$-decomposition of $X$. For each $j \in J_3$, we set $P^j_n = P_n \cap X_j$ for $n =1,2,3$, and we have
\begin{align*}
X_j & =  X\cap X_j = (P_1\cup P_2\cup P_3)\cap X_j\\
    & = (P_1\cap X_j)\cup (P_2\cap X_j)\cup (P_3\cap X_j)\\
    & = P^j_1 \cup P^j_2 \cup P^j_3.
\end{align*}
Moreover,
\begin{itemize}
\item $P_3^j= \{x \in X_j: \lambda_j(x)=x\}$;
\item if $x \in P_1^j$ ($P_2^j$), then $\lambda_j(x) \in P_2^j$ ($P_1^j$);
\item if $x \in P_1^j$ ($P_2^j$) and $y \leq x$ ($x \leq y$), then $y \in P_1^j$ ($P_2^j$).
\end{itemize}
Therefore, $(P_1^j,P_2^j,P_3^j)$ is a $\lambda_j$-decomposition of $X_j$ for each $j \in J_3$. Consider the subset $J_3'=\{j\in J_3 : P^j_3=\emptyset\}$ of $J_3$. By \cite[Theorem~5.4]{BFS14}, for each $j \in J_3'$ there are only two equivalence classes of involutions on $\FIj$ that induce $\lambda_j$, via inner automorphism. On the other hand, if $j \in J_3\setminus J_3'$, then $P_3^j \neq \emptyset$. In this case, by \cite[Theorem~5.5]{BFS14}, the number of equivalence classes of involutions on $\FIj$ that induce $\lambda_j$, via inner automorphism, is equal to $|S_K|^{|P_3^j|-1}$, where $S_K = K^*/(K^*)^2$. Thus, by Lemma~\ref{bij_classes},
$$\left|\dfrac{\Inv_\lambda(\FI)}{\thickapprox}\right| = \displaystyle\prod_{j \in J_3}\left|\frac{\Inv_{\lambda_j}(\FIj)}{\thickapprox}\right| =2^{|J_3'|} \displaystyle\prod_{j \in J_3 \setminus J_3'}|S_K|^{|P_3^j|-1}.$$
Therefore we have the following result.

\begin{theorem}\label{teonumber}
The number of equivalence classes of involutions on $\FI$ that induce $\lambda$, via inner automorphism, is equal to
$2^{|J_3'|} \displaystyle\prod_{j \in J_3 \setminus J_3'}|S_K|^{|P_3^j|-1}$.
\end{theorem}

\subsection{General classification of involutions}

The general classification of involutions on $\FI$ follows directly from the classification via inner automorphism, by the following theorem, whose proof is the same as in \cite[Theorem~5.6]{BFS14}.

\begin{theorem}~\label{teogeralclassiinvolu}
The involutions $\rho_1$ and $\rho_2$ on $\FI$ are equivalent if and only if there exists an automorphism $\alpha$ of $X$ such that $\rho_1$ and $\widehat{\alpha} \circ \rho_2 \circ \widehat{\alpha}^{-1}$ are equivalent via inner automorphism.
\end{theorem}
\begin{proof}
The involutions $\rho_1$ and $\rho_2$ on $\FI$ are equivalent if and only if there exists an automorphism $\Phi$ of $\FI$ such that $\rho_1 = \Phi \circ \rho_2 \circ \Phi^{-1}$. By Theorem~\ref{DecoAutFIP}, $\Phi= \Psi \circ \widehat{\alpha}$ where $\Psi \in \Inn(\FI)$ and $\alpha$ is an automorphism of $X$. Therefore, $\rho_1$ and $\rho_2$ are equivalent if and only if $\rho_1 = \Psi \circ \widehat{\alpha} \circ  \rho_2 \circ \widehat{\alpha}^{-1}\circ \Psi^{-1}$.
\end{proof}

As in \cite[p.1953]{BFS11}, we consider the equivalence relation $\sim$ on the set of all involutions on $X$ as follows:
\begin{center}
    $\lambda \sim \mu \Leftrightarrow \alpha\circ \lambda = \mu \circ \alpha$ for some automorphism $\alpha$ of $X$.
\end{center}
\begin{remark} \label{obs4}
Let $\lambda$ and $\alpha$ be an involution and an automorphism of $X$, respectively, and $u\in U(\FI)$. If $\rho=\Psi_u\circ \rho_{\lambda}$ is an involution on $\FI$, then $\widehat{\alpha}\circ\rho\circ\widehat{\alpha}^{-1}=\Psi_{\widehat{\alpha}(u)}\circ\rho_{\alpha\circ\lambda\circ\alpha^{-1}}$.
In particular, $\alpha\circ\lambda\circ\alpha^{-1}$ is the involution on $X$ induced by the involution $\widehat{\alpha}\circ\rho\circ\widehat{\alpha}^{-1}$.
\end{remark}

\begin{corollary} \label{corogeralclassi}
\begin{enumerate}
\item[(i)] If $\rho_1$ and $\rho_2$ are equivalent involutions on $\FI$, then $\lambda_{\rho_1} \sim  \lambda_{\rho_2}$.
\item[(ii)] If $\lambda_1 \sim \lambda_2$ and $\rho$ is an involution on $\FI$ that induces $\lambda_1$ on $X$, then $\rho$ is equivalent to some involution $\eta$ that induces $\lambda_2$.
\end{enumerate}
\end{corollary}
\begin{proof}
The proof is the same as in \cite[Corollary~27]{BFS11}, just replacing Theorem~25, Corollary~6, and Proposition~26 of \cite{BFS11} by Theorem~\ref{teogeralclassiinvolu}, \cite[Theorem~5.1]{BFS14}, and Remark~\ref{obs4}, respectively.
\end{proof}

We have seen that if $\lambda$ is an involution on $X$ such that $\lambda(X_j)\neq X_j$ for all $j\in J$, then every involution on $\FI$ that induces $\lambda$ is equivalent to $\rho_{\lambda}$ (via inner automorphism) (Theorem~\ref{teoeqtoRholambda}). On the other hand, if $\lambda(X_i)= X_i$ for some $i\in J$, then the equivalence via inner automorphism of two involutions on $\FI$ that induce $\lambda$ is given by the restriction to the set $J_3= \{j \in J : \lambda(X_j)=X_j\}$ (Theorem~\ref{teoiffinnerJ3}). Finally, we will see that the general classification, in the latter case, is also given by the restriction to $J_3$ (Theorem~\ref{teoiffgeral}).

\begin{lemma}\label{lemmaalphaJ3}
Let $\lambda$ be an involution on $X$ such that $J_3\neq \emptyset$ and $\alpha$ an automorphism of $X$ such that $\alpha\circ\lambda=\lambda\circ\alpha$. Then $\alpha(X_{J_3})=X_{J_3}$.
\end{lemma}
\begin{proof}
Let $j\in J_3$. By Remark~\ref{rest_maps_posets}, there exists $i\in J$ such that $\alpha(X_j)=X_i$, and by \eqref{eqj3}, $\lambda(X_j)=X_j$. So
$$X_i=\alpha(X_j)=\alpha(\lambda(X_j))=\lambda(\alpha(X_j))=\lambda(X_i)$$
and then $i\in J_3$. Thus, $\alpha(X_j)=X_i\subseteq X_{J_3}$. Therefore, $\alpha(X_{J_3})\subseteq X_{J_3}$. On the other hand, since $\alpha^{-1}$ is an automorphism such that $\alpha^{-1}\circ\lambda=\lambda\circ\alpha^{-1}$, then  $\alpha^{-1}(X_{J_3})\subseteq X_{J_3}$.
\end{proof}

\begin{theorem}\label{teoiffgeral}
Let $\lambda$ be an involution on $X$ such that $J_3\neq \emptyset$. Let $\rho, \eta \in \Inv_\lambda(\FI)$. Then $\rho$ and $\eta$ are equivalent if and only if the involutions $\rho_{J_3}$ and $\eta_{J_3}$ on $\FIJFix$ are equivalent.
\end{theorem}
\begin{proof}
Suppose $\rho$ and $\eta$ are equivalent. By Theorem~\ref{teogeralclassiinvolu}, there exists an automorphism $\alpha$ of $X$ such that $\rho$ and $\widehat{\alpha}\circ\eta\circ\widehat{\alpha}^{-1}$ are equivalent via inner automorphism. Thus, by \cite[Theorem~5.1]{BFS14} and Remark~\ref{obs4}, $\lambda=\alpha \circ \lambda \circ \alpha^{-1}$, and then $\alpha(X_{J_3})=X_{J_3}$ by Lemma~\ref{lemmaalphaJ3}. It follows from Theorem~\ref{teoiffinnerJ3}, Proposition~\ref{prop2}~(iii) and \eqref{eqfixPJ3} that $\rho_{J_3}$ and $\widehat{\alpha_{J_3}}\circ\eta_{J_3}\circ\widehat{\alpha_{J_3}}^{-1}$ are equivalent via inner automorphism. Therefore, $\rho_{J_3}$ and $\eta_{J_3}$ are equivalent by Theorem~\ref{teogeralclassiinvolu}.

Conversely, suppose $\rho_{J_3}$ and $\eta_{J_3}$ are equivalent. By Theorem~\ref{teogeralclassiinvolu}, there exists an automorphism $\beta$ of $X_{J_3}$ such that $\rho_{J_3}$ and $\widehat{\beta} \circ \eta_{J_3}\circ \widehat{\beta}^{-1}$ are equivalent via inner automorphism. By \cite[Theorem~5.1]{BFS14}, Definition~\ref{def1}, and Remark~\ref{obs4},
\begin{equation}\label{eqbetalambJ3}
    \lambda_{J_3}= \beta \circ \lambda_{J_3} \circ \beta^{-1}.
\end{equation}
Consider the automorphism $\alpha:X \to X$ defined by
\begin{equation} \label{eqdefalpha}
    \alpha(x) = \begin{cases}
    \beta(x) & \text{ if } x \in X_{J_3}\\
    x & \text{ if } x \not\in X_{J_3}
    \end{cases}.
\end{equation}
Then $\alpha(X_{J_3})=\beta(X_{J_3})=X_{J_3}$ and $\alpha_{J_3}=\beta$. Thus, by Proposition~\ref{prop2}~(iii),
$$(\widehat{\alpha}\circ \eta \circ \widehat{\alpha}^{-1})_{J_3} = \widehat{\alpha_{J_3}}\circ \eta_{J_3} \circ \widehat{\alpha_{J_3}}^{-1}=\widehat{\beta} \circ \eta_{J_3} \circ \widehat{\beta}^{-1}.$$
It follows that $\rho_{J_3}$ and $(\widehat{\alpha}\circ \eta \circ \widehat{\alpha}^{-1})_{J_3}$ are equivalent via inner automorphism. Moreover, by \eqref{eqbetalambJ3} and \eqref{eqdefalpha}, $\lambda=\alpha\circ\lambda\circ\alpha^{-1}$, therefore $\rho$ and $\widehat{\alpha}\circ \eta \circ \widehat{\alpha}^{-1}$ induce the same involution on $X$, by Remark~\ref{obs4}. Thus, by Theorem~\ref{teoiffinnerJ3}, $\rho$ and $\widehat{\alpha}\circ\eta\circ\widehat{\alpha}^{-1}$ are equivalent via inner automorphism, whence $\rho$ and $\eta$ are equivalent, by Theorem~\ref{teogeralclassiinvolu}.
\end{proof}







\section*{Acknowledgments}
The second author was financed in part by Coordena\c{c}\~{a}o de Aperfei\c{c}oamento de Pessoal de N\'{i}vel Superior - Brasil (CAPES) - Finance 001.


\begin{thebibliography}{99}

\bibitem{BFS11} Brusamarello, R., Fornaroli, E.~Z., Santulo Jr., E.~A., Classification of involutions on incidence algebras, {\it Comm.~Algebra} {\bf 39} (2011) 1941--1955. DOI: https://doi.org/10.1080/00927872.2010.480958

\bibitem{BFS12} Brusamarello, R., Fornaroli, E.~Z., Santulo Jr., E.~A., Anti-automorphisms and involutions on (finitary) incidence algebras, {\it Linear Multilinear Algebra} {\bf 60} (2012) 181--188. DOI: https://doi.org/10.1080/03081087.2011.576393

\bibitem{BFS14} Brusamarello, R., Fornaroli, E.~Z., Santulo Jr., E.~A., Classification of involutions on finitary incidence algebras, {\it Internat.~J.~Algebra Comput.} \textbf{24}(8) (2014), 1085--1098. DOI: https://doi.org/10.1142/S0218196714500477

\bibitem{BMUL} Brusamarello, R., Fornaroli, E.~Z., Santulo Jr., E.~A., Multiplicative automorphisms of incidence algebras, \textit{Comm.~Algebra} \textbf{43}(2) (2015), 726--736. DOI: https://doi.org/10.1080/00927872.2013.847951

\bibitem{BL} Brusamarello, R., Lewis, D.~W., Automorphisms and involutions on incidence algebras, \textit{Linear Multilinear Algebra} \textbf{59}(11) (2011), 1247--1267. DOI: https://doi.org/10.1080/03081087.2010.496113

\bibitem{DKL} Di Vincenzo, O.~M., Koshlukov, P., La Scala, R., Involutions for upper triangular matrix algebras, \textit{Adv.~in Appl.~Math.} \textbf{37}(4) (2006), 541--568. DOI: https://doi.org/10.1016/j.aam.2005.07.004

\bibitem{ER2} Fornaroli, E.~Z., Pezzott, R.~E.~M., Anti-isomorphisms and involutions on the idealization of the incidence space over the finitary incidence algebra, \textit{Linear Algebra Appl.} \textbf{637} (2022), 82--109. DOI: https://doi.org/10.1016/j.laa.2021.12.005

\bibitem{Jacobson} Jacobson, N., {\it Finite-Dimensional Division Algebras over Fields}, Springer-Verlag, Berlin, 1996.

\bibitem{AutK} Khripchenko, N.~S., Automorphisms of finitary incidence rings, {\it Algebra Discrete Math.} \textbf{9}(2) (2010), 78--97.


\bibitem{KN} Khripchenko, N.~S., Novikov, B.~V., Finitary incidence algebras, {\it Comm.~Algebra} \textbf{37}(5) (2009), 1670--1676. DOI: https://doi.org/10.1080/00927870802210019

\bibitem{Knus} Knus, M-A., Merkurjev, A., Rost, M., Tignol, J-P., \textit{The Book of Involutions}, American Mathematical Society Colloquium Publications, vol.~44, American Mathematical Society, Providence, RI, 1998.

\bibitem{Rota} Rota, G.-C., On the foundations of combinatorial theory. I. Theory of M{\"o}bius functions, \textit{Z.~Wahrscheinlichkeitstheorie und Verw.~Gebiete} \textbf{2}(4) (1964), 340--368. DOI: https://doi.org/10.1007/BF00531932

\bibitem{Scharlau} Scharlau, W., Automorphisms and involutions of incidence algebras, \textit{Lectures Notes in Mathematics} \textbf{488} (1975), 340--350.
DOI: https://doi.org/10.1007/BFb0081233

\bibitem{Spiegel05} Spiegel, E., Involutions in incidence algebras, \textit{Linear Algebra Appl.} \textbf{405} (2005), 155--162. DOI:
https://doi.org/10.1016/j.laa.2005.03.003

\bibitem{Spiegel08} Spiegel, E., Upper-triangular embeddings of incidence algebras with involution, \textit{Comm.~Algebra} \textbf{36}(5) (2008), 1675--1681. DOI: https://doi.org/10.1080/00927870801937224

\bibitem{SO97} Spiegel, E., O'Donnell, C.~J.,  \textit{Incidence Algebras}, Monographs and Textbooks in Pure and Applied Mathematics, vol.~206, Marcel Dekker, Inc., New York, 1997.

\end{thebibliography}
\end{document}